\documentclass[twoside,a4paper,10pt]{amsart}

\usepackage{amsmath}
\usepackage{times}
\usepackage{hyperref}
\usepackage{dsfont}
\usepackage{epsfig}
\usepackage{color}
\usepackage{bm}
\usepackage[euler-digits]{eulervm}
\numberwithin{equation}{section}

\hypersetup{
    pdftoolbar=true,
    pdfmenubar=true,
    pdffitwindow=false,
    pdfstartview={FitH},
    pdftitle={Functions on the sphere with critical points in pairs and orthogonal geodesic chords},
    pdfauthor={Roberto Giamb\`o, Fabio Giannoni and Paolo Piccione},
    pdfsubject={},
    pdfkeywords={}
    pdfnewwindow=true,
    colorlinks=true, 
    linkcolor=blue,
    citecolor=blue,
    urlcolor=blue,
}

\title[Functions on the sphere with critical points in pairs]%
{Functions on the sphere with critical points in pairs and orthogonal geodesic chords}

\author[R. Giamb\`o]{Roberto Giamb\`o}
\author[F. Giannoni]{Fabio Giannoni}
\address{Dipartimento di Matematica e Informatica,\hfill\break\indent
Universit\`a di Camerino, Italy}
\email{roberto.giambo@unicam.it, fabio.giannoni@unicam.it}
\author[P. Piccione]{Paolo Piccione}
\address{Departamento de Matem\'atica \hfill\break\indent Universidade de S\~ao Paulo \hfill\break\indent
Rua do Mat\~ao, 1010 \hfill\break\indent 05508-090 S\~ao Paulo, SP, Brazil}
\email{piccione.p@gmail.com}

\urladdr{http://www.ime.usp.br/\~{}piccione}
\date{March 6th, 2015}


\begin{document}


\theoremstyle{plain}\newtheorem{teo}{Theorem}[section]
\theoremstyle{plain}\newtheorem{prop}[teo]{Proposition}
\theoremstyle{plain}\newtheorem{lem}[teo]{Lemma}
\theoremstyle{plain}\newtheorem{cor}[teo]{Corollary}
\theoremstyle{definition}\newtheorem{defin}[teo]{Definition}
\theoremstyle{remark}\newtheorem{rem}{Remark}

\theoremstyle{remark}\newtheorem{example}{Example}


\theoremstyle{plain}
\newtheorem*{teoA}{Theorem A}
\newtheorem*{teoB}{Theorem B}
\newtheorem*{teoC}{Theorem C}
\newtheorem*{corD}{Corollary D}

\theoremstyle{definition}
\newtheorem*{defin_intro}{Definition}


\begin{abstract}
Using an estimate on the number of critical points for a \emph{Morse-even} function on the sphere $\mathbb S^m$, $m\ge1$, we prove a multiplicity result for orthogonal geodesic chords in Riemannian manifolds with boundary that are diffeomorphic to Euclidean balls. This yields also
 a multiplicity result for brake orbits in a potential well.

\end{abstract}

\maketitle
\begin{section}{Introduction}
The topology of spheres does not allow good estimates on the number of critical points of smooth functions. In fact, given the fact that any
$C^1$-function on a compact manifold admits at least two critical points, Reeb's theorem characterizes spheres as the only compact manifold that admits functions with \emph{exactly} $2$ critical points (see for instance \cite[Theorem~4.1, p.\ 25]{Milnor_Morse}). However, functions having some type of symmetry tend to have more than two critical points. For instance, if $f:\mathbb S^m\to\mathds R$ is even, i.e., $f(x)=f(-x)$ for all $x$, then $f$ defines a function $\widetilde f$ on the projective space $\mathds RP^m$, that must have at least $m+1$ distinct critical points, by Lusternik--Schnirelman (or Morse) theory. Thus, the original function $f$ must have at least $m+1$ distinct pairs of antipodal critical points.
\subsection{Morse-even functions}
Motivated by an example in classical Riemannian geometry, that will be described below, in this paper we prove that the same estimate on the number of critical points holds in a slightly more general situation, when resorting to the Lusternik--Schnirelman category of the projective space is not possible. More precisely, we will consider functions whose critical points come in pairs with the same Morse index.
\begin{defin_intro}
Let $\mathcal M^m$ be a compact $m$-dimensional manifold. A Morse function $f:\mathcal M\to\mathds R$
is said to be \emph{Morse-even} if for every $k=0,\ldots,m$, the set of critical points of $f$ having Morse
index equal to $k$ is an even number.
\end{defin_intro}

The starting observation for Morse-even functions on spheres is the following:
\begin{teoA}
Let $\mathcal M$ be a smooth manifold that is homeomorphic to an $m$-sphere, and let $f:\mathcal M\to\mathds R$ be a Morse-even function. Then, $f$ has critical points of arbitrary Morse index in $\{0,1,\ldots,m\}$. In particular, $f$ admits at least $(m+1)$ distinct pairs of critical points.
\end{teoA}
\begin{proof}
The Poincar\'e polynomial (with integer coefficients) of the $m$-sphere is given by $P_m(\lambda)=1+\lambda^m$. By Morse theory, there exists
a polynomial $Q(\lambda)=a_0+a_1\lambda+\ldots+a_{m-1}\lambda^{m-1}$  with nonnegative integer coefficients such that:
\begin{equation}\label{eq:MorsePol}
1+\lambda^m+(1+\lambda)Q(\lambda)=(1+a_0)+(a_0+a_1)\lambda+\ldots+(a_{m-2}+a_{m-1})\lambda^{m-1}+(1+a_{m-1})\lambda^m
\end{equation}
is the Morse polynomial $M_f(\lambda)$ of $f$. Recall that $M_f(\lambda)=\sum_{k=0}^m\kappa_k\lambda^k$, where
$\kappa_k$ is the number of critical points of $f$ whose Morse index is equal to $k$. Since $f$ is Morse-even, then each $\kappa_k$ is an even number.
Using \eqref{eq:MorsePol}, one shows by an elementary argument that all the $a_i$ are non-zero and odd, $i=0,\ldots,m-1$.
Thus, all the $\kappa_k$ are positive, $k=0,\ldots,m$.
\end{proof}
\begin{rem}\label{rem:intro}
The above result has a natural generalization to Morse-even functions defined on a connected, orientable compact manifold $\mathcal M$ whose Betti numbers $\beta_k(\mathcal M)$ are even for $k=1,\ldots,m-1$.
If $f:\mathcal M\to\mathds R$ is a Morse-even function on such a manifold, then for all $k=1,\ldots,m-1$, $f$ admits a 
number strictly larger than $\frac12\beta_k(M)$ of pairs of critical points having Morse index equal to $k$
(this statement holds trivially also for $k=0$ and $k=m$).
A proof of this is obtained readily from the Morse relations, given by the equality:
\[\sum_{k=0}^m\beta_kx^k+(1+x)\sum_{k=1}^{m-1}a_kx^k=\sum_{k=0}^m\kappa_kx^k,\]
i.e.:
\[\kappa_0=1+a_0,\quad \kappa_k=\beta_k+a_k+a_{k-1},\ \text{for $k=1,\ldots,m-1$},\quad\kappa_m=1+a_m,\]
for some integer coefficients $a_k\ge0$.
The parity of the coefficients implies $a_0,a_m\ge1$, and $a_k+a_{k-1}\ge2$ for $k=1,\ldots,m-1$.

The condition on the parity of the Betti numbers is satisfied by a large class of manifolds.
For instance, when $m=2$, the condition $\beta_1$ even is satisfied by every compact oriented surface $\Sigma$.
Namely, in this case $\beta_1(\Sigma)=2\,\mathrm{gen}(\Sigma)$ (here $\mathrm{gen}(\Sigma)$ denotes the genus of $\Sigma$).
It is not hard to classify the homeomorphism classes of simply connected manifolds satisfying the condition in low dimensions.
When $m=3$, if $\mathcal M$ is simply connected then $\beta_1(\mathcal M)=0$ (and thus also $\beta_2(\mathcal M)=0$). By the Poincar\'e conjecture, the unique simply connected compact manifold that satisfies the assumptions is the $3$-sphere. Many interesting cases of even Betti numbers are found in dimension $4$. In this case, for a simply connected manifold it suffices to require that $\beta_2$ is even. Recall that in dimension $4$, the second Betti number $\beta_2$ is additive by connected sums, i.e., given $4$-manifolds $\mathcal M_1$ and $\mathcal M_2$,
then $\beta_2(\mathcal M_1\#\mathcal M_2)=\beta_2(\mathcal M_1)+\beta_2(\mathcal M_2)$.
Recall also that $\beta_2(\mathds CP^2)=1$ and $\beta_2(\mathbb S^2\times\mathbb S^2)=2$.
Thus, connected sums of any number of copies of $\mathbb S^2\times\mathbb S^2$ and any even number of copies of $\mathds CP^2$ have even $\beta_2$.
\end{rem}
\subsection{Orthogonal geodesic chords}
Our motivation for the result of Theorem~A comes from a classical problem in Riemannian geometry, which consists in finding lower estimates on the number of
geodesics departing and arriving orthogonally to the boundary of a compact Riemannian manifold.
These objects are called \emph{orthogonal geodesic chords} (OGC).
It is interesting to observe that there are manifolds diffeomorphic to Euclidean balls that have no OGC's,
see for instance \cite{bos}.
Orthogonal geodesics chords for metrics in a ball have a special interest in the case where the boundary $\mathbb S^m$ of $B^{m+1}$ is strictly concave. Namely, in this situation, a multiplicity result for OCG's yields an analogous multiplicity result 
for brake orbits of natural Hamiltonians or Lagrangian in a potential well (see for instance~\cite{cag}).

In order to apply Theorem~A to obtain information on the number of OGC's, let us consider the following situation. Given a compact Riemannian manifold
$(M,g)$ with boundary $\partial M$, let $\exp$ denote the corresponding exponential map. Let $\vec\nu$ be the unit normal field along $\partial M$ pointing inwards. The metric $g$ will be said to be \emph{regular} (with respect to $\partial M$) if there exists a (necessarily smooth)
function $s_g:\partial M\to\left]0,+\infty\right[$ such that, for any $p\in\partial M$, the geodesic $\big[0,s_g(p)\big]\ni t\mapsto\exp_p(t\cdot \vec\nu_p)$ meets transversally $\partial M$ at $t=s_g(p)$. In this situation, we will call $s_g$ the
\emph{crossing time function} of the metric $g$.

It is easy to see that the set of regular metrics on a given manifold with boundary $M$, that will be denoted by 
$\mathrm{Reg}(M)$, is open
in the $C^1$-topology.
We also define \emph{non-focal} a metric on $M$ for which there are no $\partial M$-focal points along $\partial M$.
Also in this case, it is not hard to show that non-focal metrics form an open subset in the $C^2$-topology, see Proposition~\ref{thm:nofocalpoints}.

Let us denote by $\mathrm{Reg}_*(M)$ \emph{the set of non-focal regular metrics on $M$.}
Our main interest is in the case when $M$ is diffeomorphic to the unit ball $B^{m+1}$ in the Euclidean space $\mathds R^{m+1}$. In this case, the set of regular and non-focal metrics is an open subset of all Riemannian metrics containing, for instance, the set of radially symmetric metrics,
see Corollary~\ref{thm:perturbradsym}.

As an application of Theorem~A, we prove the following:
\begin{teoB}
For a \emph{generic}\footnote{Here generic is in meant in the topological sense. A subset of a topological set is generic if it contains a residual set (countable intersection of open dense subsets).}  set of metrics $g$ in $\mathrm{Reg}_*(B^{m+1})$, there are at least $(m+1)$ distinct orthogonal geodesic chords in $(B^{m+1},g)$.
\end{teoB}
More precisely, the result of Theorem~B holds for all metrics in $\mathrm{Reg}_*(B^{m+1})$ for which every OGC is nondegenerate in an appropriate sense, see the discussion after Corollary~\ref{thm:pcritiffqcrit} and Section~\ref{sub:3.2} for details.

When $M$ is convex and homeomorphic to the $m+1$-dimensional disk the multiplicity problem for OGC's is studied in 
\cite{bos,LustSchn}. If $M$ is concave there is an existence result  and a multiplicity result of two OGC's  (see \cite{pre}
and the references therein). To the authors' knowledge, Theorem~ B is the first result about multiplicity of OGC's without convexity or concavity assumption.\smallskip

For the Morse-theoretical aspects in the proof of Theorem~B, one of the key ingredients will be an index theorem for orthogonal geodesic chords (see Corollary~\ref{thm:MIT} and Corollary~\ref{thm:Morse+equaindex}), in the formulation given in \cite{PicTau99}. This result, together with a stability result for focal points proved in \cite{MerPicTau02}, is used to prove that the crossing time function $s_g$ is even-Morse, providing the desired link between even-Morse functions and orthogonal geodesic chords.

\subsection{Brake orbits of Lagrangian systems}
\label{sub:brakehom}
The result of Theorem~B can be applied to prove a new multiplicity
result for brake orbits, as illustrated below. We will present here a Lagrangian formulation of the brake orbits problem. An equivalent formulation can be given for periodic solutions of Hamiltonian systems, via Legendre transform.

Let $(\mathfrak M^{m+1},\mathfrak g)$ be a Riemannian manifold (without boundary), representing the configuration space of some dynamical systems, and let $V:\mathfrak M\to\mathds R$ be a smooth function, representing the potential energy of some conservative force acting on the system. One looks for periodic solutions $x:[0,T]\to\mathfrak M$ of the Lagrangian systems:
\begin{equation}\label{eq:lagrsystem}
\tfrac{\mathrm D}{\mathrm dt}\dot x=-\nabla V,
\end{equation}
where $\tfrac{\mathrm D}{\mathrm dt}$ denotes the covariant derivative of the Levi--Civita connection of $\mathfrak g$ for vector fields along $x$, and $\nabla V$ is the gradient of $V$. Solutions of \eqref{eq:lagrsystem} satisfy the conservation of energy law $\frac12\mathfrak g(\dot x,\dot x)+V(x)=E$, where $E$ is a real constant called the \emph{energy} of the solution $x$. It is a classical problem to give estimate of periodic solutions of \eqref{eq:lagrsystem} having a fixed value of the energy $E$. This problem has been, and still is, the main topic of a large amount of literature, also for autonomous Hamiltonian systems, 
see for instance \cite{LiuLong,LZ,Long,rab} an the references therein. We will give here a very short account of a geometric approach to periodic solutions of \eqref{eq:lagrsystem}

By the classical Maupertuis principle, solutions of \eqref{eq:lagrsystem} having energy $E$ are, up to a parameterization, geodesics in the conformal metric:
\begin{equation}\label{eq:maupertuismetric}
g_E=\big(E-V(p)\big)\cdot\mathfrak g,
\end{equation}
defined in the closed $E$-sublevel $M_E=V^{-1}\big(\left]-\infty,E\right]\big)$ of $V$. Observe that, in fact, $g_E$ degenerates on the boundary $\partial M_E=V^{-1}(E)$. Among all periodic solutions of \eqref{eq:lagrsystem}, historical importance is given to a special class called \emph{brake orbits}; these are ``pendulum-like'' solutions, that oscillate with constant frequency along a trajectory that joins two endpoints lying in $V^{-1}(E)$. Thus, brake orbits correspond to $g_E$-geodesics
in $M_E$ with endpoints in $\partial M_E$, or, more precisely, to $g_E$-geodesics
$\gamma:\left]0,T\right[\to V^{-1}\big(\left]-\infty,E\right[\big)$, with $\lim\limits_{t\to0^-}\gamma(t)$ and $\lim\limits_{t\to T^-}\gamma(t)$ in $\partial M_E$.

For such degenerate situation, it has been proved in \cite{cag} that, if $E$ is a regular value of the function $V$ (which implies in particular that $\partial M_E$ is a smooth hypersurface of $\mathfrak M$), then $g_E$ defines a distance-to-the-boundary function $\mathrm{dist}_E:M_E\to\left[0,+\infty\right[$ which is smooth in the interior of $M_E$ and extends continuously to  $0$ on the boundary $\partial M_E$. Moreover, if $\delta>0$ is small enough, then any OGC in the Riemannian manifold $M=\mathrm{dist}_E^{-1}\big(\left[\delta,+\infty\right[\big)$ endowed with the metric $g_E$ (which is now non-singular) can be extended uniquely to a $g_E$-geodesic $\gamma$ in $M_E$ with endpoints in $\partial M_E$, as above. In conclusion, any result on multiplicity of OGC's can be reformulated to a multiplicity result for brake orbits at level  a fixed regular energy level of a conservative dynamical system.
\medskip

A very famous conjecture due to Seifert, see \cite{seifert}, asserts that, given a Lagrangian system as in \eqref{eq:lagrsystem}, if the sublevel $V^{-1}\big(\left]-\infty,E]\right]\big)$ is homeomorphic to an $(m+1)$-ball $B^{m+1}$, then there are at least $m+1$ distinct brake orbits. This estimate is known to be sharp, i.e., there are examples of Lagrangian systems having energy levels homeomorphic to an $(m+1)$-ball and admitting exactly $m+1$ distinct brake orbits. A proof of Seifert's conjecture in its full generality is still open, but the question has been solved affirmatively in some cases. When $V$ is even and convex, multiplicity results  are obtained  in \cite{LZ,LZZ,Z1,Z2,Z3}.
In particular in \cite{LZ} there is the proof of the Seifert conjecture for euclidean metrics and even and convex potentials.

When the $E$-sublevel $V^{-1}\big(\left]-\infty,E\right]$ has the topology of the annulus,
the multiplicity of brake orbits is studied in depth in \cite{arma} and \cite{JDE}.
\smallskip

Theorem~B yields the following contribution to Seifert's conjecture:
\begin{teoC}
Let $E$ be a regular value of $V$, such that $V^{-1}\big(\left]-\infty,E\right]\big)$ is homemorphic to an $(m+1)$-ball
$B^{m+1}$.
Assume that there exists $\delta>0$ sufficiently small such that the Riemannian manifold $M=\mathrm{dist}_E^{-1}\big(\left[\delta,+\infty\right[\big)$ endowed with the metric $g_E$ satisfies the assumptions of Theorem~B. Then there are at least $m+1$ distinct brake orbits of energy $E$ for the Lagrangian system~\eqref{eq:lagrsystem}.\qed
\end{teoC}

In particular, from Proposition~\ref{thm:nofocalpoints} below we obtain the following:
\begin{corD}\label{thm:seifertradially}
Seifert conjecture is generically true in a $C^2$-open set of potentials $V$ that contains the  ones that are \emph{rotationally symmetric} at level $E$.\qed
\end{corD}
By a potential $V$ rotationally symmetric at level $E$ we mean that there is a continuous action of the rotation group $\mathrm{SO}(m+1)$ on $V^{-1}\big(\left]-\infty,E\right]\big)$ that makes the sublevel equivariantly diffeomorphic to the Euclidean ball $B^{m+1}$ with the canonical $\mathrm{SO}(m+1)$-action, and such that $V$ is constant along the orbits of this action.
\end{section}
\begin{section}{Variational Theory for OGC's}
Let us show how obtain multiple orthogonal geodesic chords in a compact Riemannian manifold with boundary $(M,g)$,
using a function on $\partial M$ whose critical points are OGC's. Next Proposition shows that, for a regular metric $g$, such function is precisely the crossing time function $s_g$. In order to prove this, let us introduce some notations.

For $p\in\partial M$, let $\mathcal S_p:T_p(\partial M)\to T_p(\partial M)$ denote the shape operator of $\partial M$ at $p$ in the normal direction $\vec\nu_p$.
Let $\exp^\perp:U\subset T(\partial M)^\perp\to M$ denote the normal exponential map of $g$ along $\partial M$;
for $p\in\partial M$, let $\gamma_p:\big[0,s_g(p)\big]\to M$ denote the geodesic $t\mapsto\exp_p(t\cdot\vec\nu_p)$.
Recall that a point $q\in M$ is a singular value of $\exp^\perp$ exactly when $q$ is \emph{focal} to $\partial M$.
If $q\in M$ is a singular value of $\exp^\perp$ and $v\in T_p(\partial M)^\perp$ is the corresponding critical point, so that
$\gamma_p(t_*)=q$ for some $t_*\in\left]0,s_g(p)\right]$, then
the kernel of $\mathrm d\exp^\perp(v)$ consists of \emph{$\partial M$-Jacobi fields} along the geodesic $\gamma_p$
that vanish at $t_*$. Recall that a Jacobi field along $\gamma_p$ is called a $\partial M$-Jacobi field if it satisfies
the initial conditions:
\begin{equation}\label{eq:ICond}
J(0)\in T_p(\partial M),\quad J'(0)+\mathcal S_p\big(J(0)\big)\in T_p(\partial M)^\perp,
\end{equation}
where $J'$ denotes the covariant derivative of $J$ along $\gamma_p$.

\begin{figure}
\begin{center}
\psfull \epsfig{file=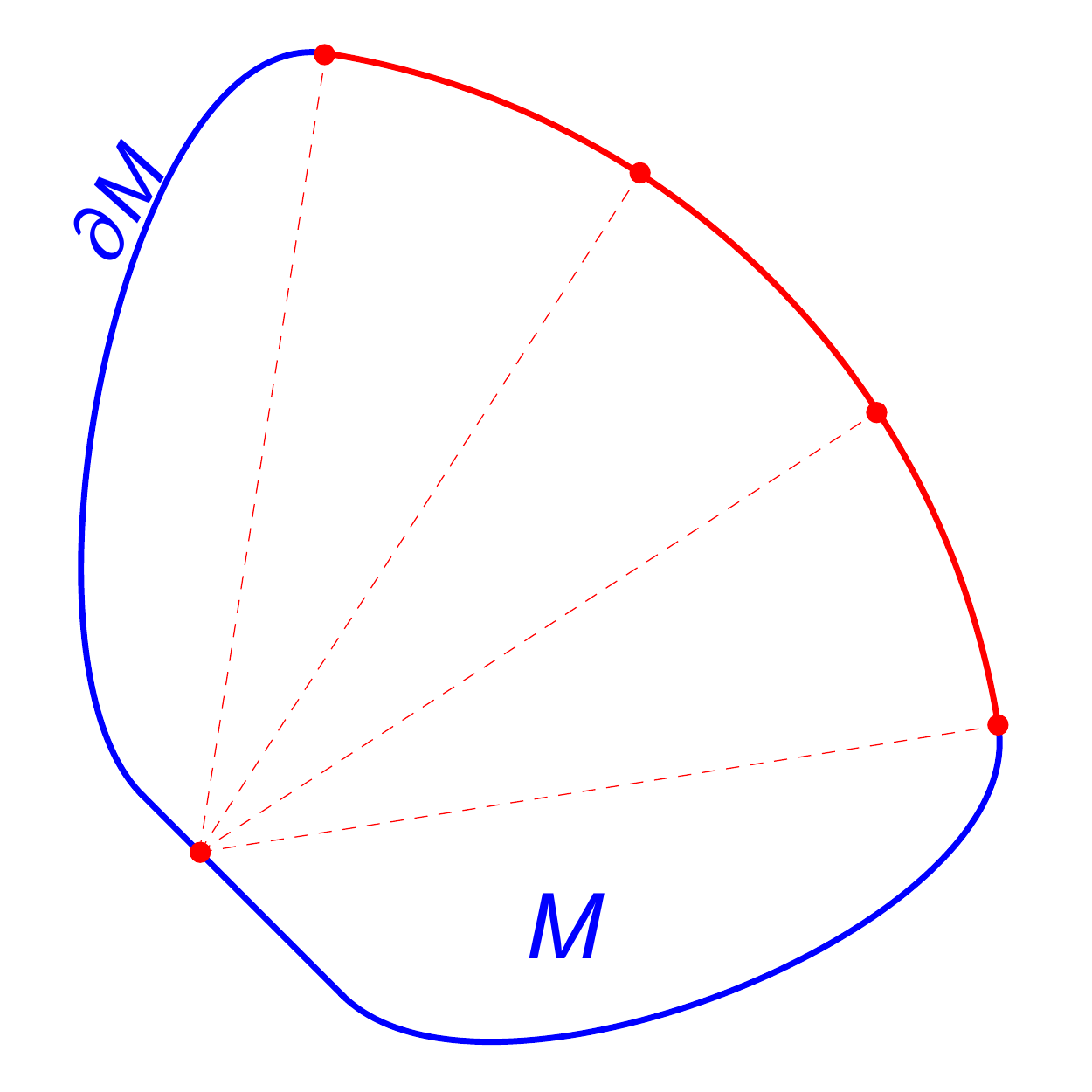, height=5cm} \caption{The picture represents a situation in which critical points of the function $s_g$ do not correspond to orthogonal geodesic chords (see Remark \ref{rem:counterex1}).
Here, $M$  is a compact subset of a Euclidean space. A portion of the boundary of $M$ (in red) is a spherical hypersurface, whose center, which is a $\partial M$-focal point, lies on $\partial M$. Orthogonal geodesic chords starting from this spherical surfaces converge and meet at the center; thus the function $s_g$ is constant in this spherical region of the boundary. However, these critical points do not correspond to OGC's in $M$.}
\label{fig:01}
\end{center}
\end{figure}

For $p\in\partial M$, denote by $q_p=\gamma_p\big(s_g(p)\big)\in\partial M$, and
let $\mathcal E_p\subset T_{q_p}M$ the image of the linear map $\mathrm d\exp^\perp\big(s_g(p)\vec\nu_p\big)$.
Equivalently:
\[\mathcal E_p=\Big\{J\big(s_g(p)\big):\text{$J$ is a $\partial M$-Jacobi field along $\gamma_p$}\Big\}.\]
When $q_p$ is not $\partial M$-focal, then $\mathcal E_p=T_{q_p}M$.
\begin{prop}\label{thm:critpointssg}
Assume $g$ regular. A point $p\in\partial M$ is critical for $s_g:\partial M\to\left]0,+\infty\right[$ if and only if the vector
$\dot\gamma_p\big(s_g(p)\big)\in T_{q_p}M$ is orthogonal to the intersection\footnote{%
We observe that, under the assumptions of Proposition~\ref{thm:critpointssg}, $q_p$ is $\partial M$-focal if and only if $\mathcal E_p\cap T_{q_p}\partial M\ne T_{q_p}\partial M$. This follows easily from the fact that $\mathcal E_p$ always contains $\dot\gamma_p\big((s_g(p)\big)$ which is transversal to $T_{q_p}\partial M$.}
$\mathcal E_p\cap T_{q_p}(\partial M)$.
In particular, if $q_p$ is not $\partial M$-focal, then $\gamma_p$ is an orthogonal geodesic chord in $M$.
\end{prop}
\begin{proof}
Since $s_g>0$, the critical points of $s_g$ coincide with those of $s_g^2$.
For $p\in\partial M$, denote by $\sigma_p:[0,1]\to M$ the affinely reparameterized geodesic $\sigma_p(t)=\gamma_p\big(s_g(p)\cdot t\big)$, $t\in[0,1]$. Integration on $[0,1]$ gives:
\begin{equation}\label{eq:sgaction}
\int_0^1g\left(\dot\sigma_p,\dot\sigma_p\right)\,\mathrm dt=s_g(p)^2,
\end{equation}
i.e., $\sigma_g^2$ can be thought of the geodesic action functional of $g$ applied to the geodesic $\sigma_p$.
Let $p\in\partial M$ be fixed, and let $\rho:\left]-\varepsilon,\varepsilon\right[\to\partial M$ be a $C^1$-curve satisfying $\rho(0)=p$ and $\dot\rho(0)=w\in T_p(\partial M)$. Then, $p$ is critical for $s_g^2$ iff for any such $\rho$ one has $\frac{\mathrm d}{\mathrm dt}\big\vert_{t=0}s_g^2\left(\rho(t)\right)=0$. Using \eqref{eq:sgaction} one has:
\begin{equation}\label{eq:firstvarsg}
\tfrac{\mathrm d}{\mathrm dt}\big\vert_{t=0}s_g^2\left(\rho(t)\right)=2\int_0^1g\left(\dot\sigma_p,J_w'\right)\,\mathrm dt,
\end{equation}
where  $J_w$ is the Jacobi field along $\sigma_p$ given by:
\[J_w(s)=\tfrac{\mathrm d}{\mathrm dt}\big\vert_{t=0}\sigma_{\rho(t)}(s),\quad s\in[0,1].\]
Keeping in mind that $J_w(0)=w\in T_p(\partial M)$ and that $\dot\sigma_p(0)\in T_p(\partial M)^\perp$, integration by parts in \eqref{eq:firstvarsg} gives:
\[\tfrac{\mathrm d}{\mathrm dt}\big\vert_{t=0}s_g^2\left(\rho(t)\right)=2g\left(\dot\sigma_p(1),J_w(1)\right).\]
It is easily seen that the map $T_p(\partial M)\ni w\mapsto J_w(1)\in\mathcal E_p\cap T_{q_p}(\partial M)$ is surjective, and from this observation the thesis follows readily.
\end{proof}
\begin{rem}\label{rem:counterex1}
It is not hard to give examples of regular metrics that do not satisfy the non-focal property, and in which critical points of $s_g$ do not correspond to orthogonal geodesic chords. See Figure \ref{fig:01}.
\end{rem}

From Proposition~\ref{thm:critpointssg}, we obtain immediately the following:
\begin{cor}
\label{thm:pcritiffqcrit}
Let $g$ be regular and non-focal.
Then, a point $p$ is critical for $s_g$ if and only if the point $q_p=\gamma_p\big(s_g(p)\big)$ is critical for $s_g$. In this case, $s_g(p)=s_g(q_p)$,
and $\gamma_{q_p}=\gamma_p$ up to orientation.\qed
\end{cor}
Proposition~\ref{thm:critpointssg} gives a first order variational principle relating orthogonal geodesic chords in $M$ to critical points of a smooth function on $\partial M$. For our purposes, we need a related second order variational principle, relating nondegeneracy and Morse index of OGC's in $M$ and critical points of $s_g$.
Let us recall a few facts from the variational theory of OGC's.

Assume that $\gamma_p:\big[0,s_g(p)\big]\to M$ is an OGC in $M$, i.e., that
$\dot\gamma_p\big(s_g(p)\big)\in T_{q_p}(\partial M)^\perp$.
The \emph{index form} along $\gamma_p$ is the symmetric bilinear form $I_p$ defined on the vector space $\mathcal V_p$ of (piecewise smooth) vector fields $V$ along $\gamma_p$ satisfying $V(0)\in T_p(\partial M)$ and $V\big(s_g(p)\big)\in T_{q_p}(\partial M)$, defined\footnote{Observe that we have a different sign convention from the standard literature in the second boundary term of $I_p$, because $\mathcal S_{q_p}$ has been defined as the shape operator in the normal direction $\vec\nu_{q_p}$, and $\vec\nu_{q_p}=-\dot\gamma_p\big(s_g(p)\big)$. Similarly, the second condition in \eqref{eq:FCond} has a sign different from the standard literature.} by:
\begin{multline}\label{eq:indexform}
I_p(V,W)=\int_0^{s_g(p)}g(V',W')+g\big(R(\gamma_p',V)\gamma_p',W\big)\,\mathrm ds\\-\Big[g\big(\mathcal S_p(V(0)),W(0)\big)+g\big(\mathcal S_{q_p}(V(s_g(p))),W(s_g(p))\big)\Big],
\end{multline}
where $R$ is the curvature tensor of $g$, chosen with the sign convention \[R(X,Y)=[\nabla_X,\nabla_Y]-\nabla_{[X,Y]}.\]
It is well known that $I_p$ is the second variation of the geodesic action functional, defined in the set of paths with endpoints in $\partial M$,  at the critical point $\gamma_p$.
The OGC $\gamma_p$ is said to be \emph{nondegenerate} if $I_p$ is a nondegenerate bilinear form on $\mathcal V_p$, i.e., if $\gamma_p$ is a nondegenerate critical point of the geodesic action functional of $M$ in the space of paths with endpoints in $\partial M$. The kernel of $I_p$ is the space of $\partial M$-Jacobi fields $J$ along $\gamma_p$ that satisfy, in addition to \eqref{eq:ICond}, the following boundary condition at $t=s_g(p)$:
\begin{equation}\label{eq:FCond}
J\big(\sigma_g(p)\big)\in T_{q_p}(\partial M),\quad J'\big(s_g(p)\big)+\mathcal S_{q_p}\big(J(s_g(p))\big)\in T_{q_p}(\partial M)^\perp.
\end{equation}
Thus, $\gamma_p$ is nondegenerate if and only if there exists no non-trivial Jacobi field $J$ along $\gamma$ satisfying \eqref{eq:ICond} and
\eqref{eq:FCond}. The \emph{Morse index} of $\gamma_p$ is the index of the symmetric bilinear form $I_p$, which is the dimension of a maximal
subspace of $\mathcal V_p$ on which $I_p$ is negative definite. Recall that this is a (finite) nonnegative integer, that can be computed in terms of some focal invariants of $\partial M$, see \cite{PicTau99} for details.
\begin{prop}\label{thm:secondorder}
Assume $g$ regular. Let $p\in\partial M$ be a critical point of $s_g$, and assume that $q_p\in\partial M$ is not $\partial M$-focal along $\gamma_p$.
Then:
\begin{itemize}
\item[(a)] $p$ is a nondegenerate critical point of $s_g$ if and only if $\gamma_p$ is a nondegenerate OGC;
\item[(b)] the Hessian of $\frac12s_g^2$ at the critical point $p$ is identified with the symmetric bilinear form:
\begin{equation}\label{eq:secdersg2}
\mathbb J_p\times\mathbb J_p\ni (J_1,J_2)\longmapsto g\Big(J_1'(s_g(p))-\mathcal S_{q_p}\big(J_1(s_g(p))\big),J_2(s_g(p))\Big)\in\mathds R,
\end{equation}
where $\mathbb J_p$ is the vector space of all $\partial M$-Jacobi fields $J$ along $\gamma_p$ that satisfy \[J\big(s_g(p)\big)\in T_{q_p}(\partial M).\] Its index is less than or equal to the Morse index of the OGC $\gamma_p$.
\end{itemize}
\end{prop}
\begin{proof}
As in Proposition~\ref{thm:critpointssg}, both statements are obtained by identifying the points $p$ of $\partial M$ with the curve $\gamma_p$, as an element of the space of curves with endpoints in $\partial M$.
Using this identification, the tangent space $T_p(\partial M)$ is identified with the space of variations of $\gamma_p$ by geodesics $\gamma_q$ that start orthogonally to $\partial M$ and arrive onto $\partial M$. This space of variations is given by the vector space $\mathbb J_p$.
The function $\frac12s_g^2$ is the restriction to the set $\{\gamma_p:p\in\partial M\}$ of the geodesic action functional of $M$. We have proved that, under our assumptions, $\gamma_p$ is a critical point of this restriction, but also a critical point of the \emph{full} geodesic action functional. Hence, the second derivative of $\frac12s_g^2$ at $p$ is given by the restriction of the index form $I_p$ to the space $\mathbb J_p$; this implies, in particular, that the Morse index of $p$ is less than or equal to the Morse index of the OGC $\gamma_p$.
Formula \eqref{eq:secdersg2} is obtained easily
using partial integration in \eqref{eq:indexform}, the Jacobi equation and \eqref{eq:ICond}. This proves (b).

The assumption that $q_p$ is not $\partial M$-focal along $\gamma_p$ implies that, as $J_2$ runs in $\mathbb J_p$, the vector $J_2(b)$
is an arbitrary vector in $T_{q_p}(\partial M)$. Hence, $I_p(J_1,J_2)=0$ for all $J_2$ if and only if $J_2$ satisfies \eqref{eq:FCond}, i.e.,
the kernel of the second derivative of $s_g^2$ at the critical point $p\in\partial M$ coincides with the kernel of the index form $I_p$. This proves statement (a).
\end{proof}
Using the Morse index theorem for geodesics between two fixed submanifold, see for instance \cite{PicTau99}, one proves the following
more precise result on the Morse index of critical points of $s_g$:
\begin{cor}\label{thm:MIT}
Under the assumption of Proposition~\ref{thm:secondorder}, the Morse index of $\frac12s_g^2$ at $p$ is equal to the Morse index of the OGC $\gamma_p$ minus the number of $\partial M$-focal points along $\gamma_p$, counted with multiplicity.
\end{cor}
\begin{proof}
It follows readily from Proposition~\ref{thm:secondorder} and the result of \cite[Theorem~2.7]{PicTau99}.
More precisely, \cite[Theorem~2.7]{PicTau99} proves that, when $q_p$ is not $\partial M$-focal along $\gamma_p$, then the space
$\mathcal V_p$ (recall that this is the domain of the index form $I_p$) is the direct sum of:
\begin{itemize}
\item the space $\mathcal V_p^1$ of (piecewise smooth) vector fields $V$ along $\gamma_p$ satisfying $V(0)\in T_p(\partial M)$,
\item the space $\mathbb J_p$ of $(\partial M)$-Jacobi fields $J$ along $\gamma_p$ satisfying $J\big(s_g(p)\big)\in T_{q_p}(\partial M)$.
\end{itemize}
Such a direct sum decomposition is $I_p$-orthogonal. Hence, the index of $I_p$ on $\mathcal V_p$, which is the Morse index of the OGC $\gamma_p$, is equal to the sum of the indices of the restriction of $I_p$ to each one of the two spaces above. The index of the restriction of $I_p$ to $\mathbb J_p$ is precisely the index of the bilinear form \eqref{eq:secdersg2}, i.e., the Morse index of $\frac12s_g^2$ at $p$.
The index of the restriction of $I_p$ to the space $\mathcal V_p^1$, by the Morse index theorem for geodesics between a submanifold and a fixed point
(see \cite[Theorem~2.5]{PicTau99}), is given by the number of $\partial M$-focal points along $\gamma_p$ counted with multiplicity.
\end{proof}
\begin{cor}\label{thm:Morse+equaindex}
Let $M$ be a compact manifold with connected boundary $\partial M$, and let $g$ be a regular and non-focal Riemannian metric on $M$.
If every OCG in $M$ is nondegenerate, then $s_g$ is an even-Morse function on $\partial M$.
\end{cor}
\begin{figure}
\begin{center}
\psfull \epsfig{file=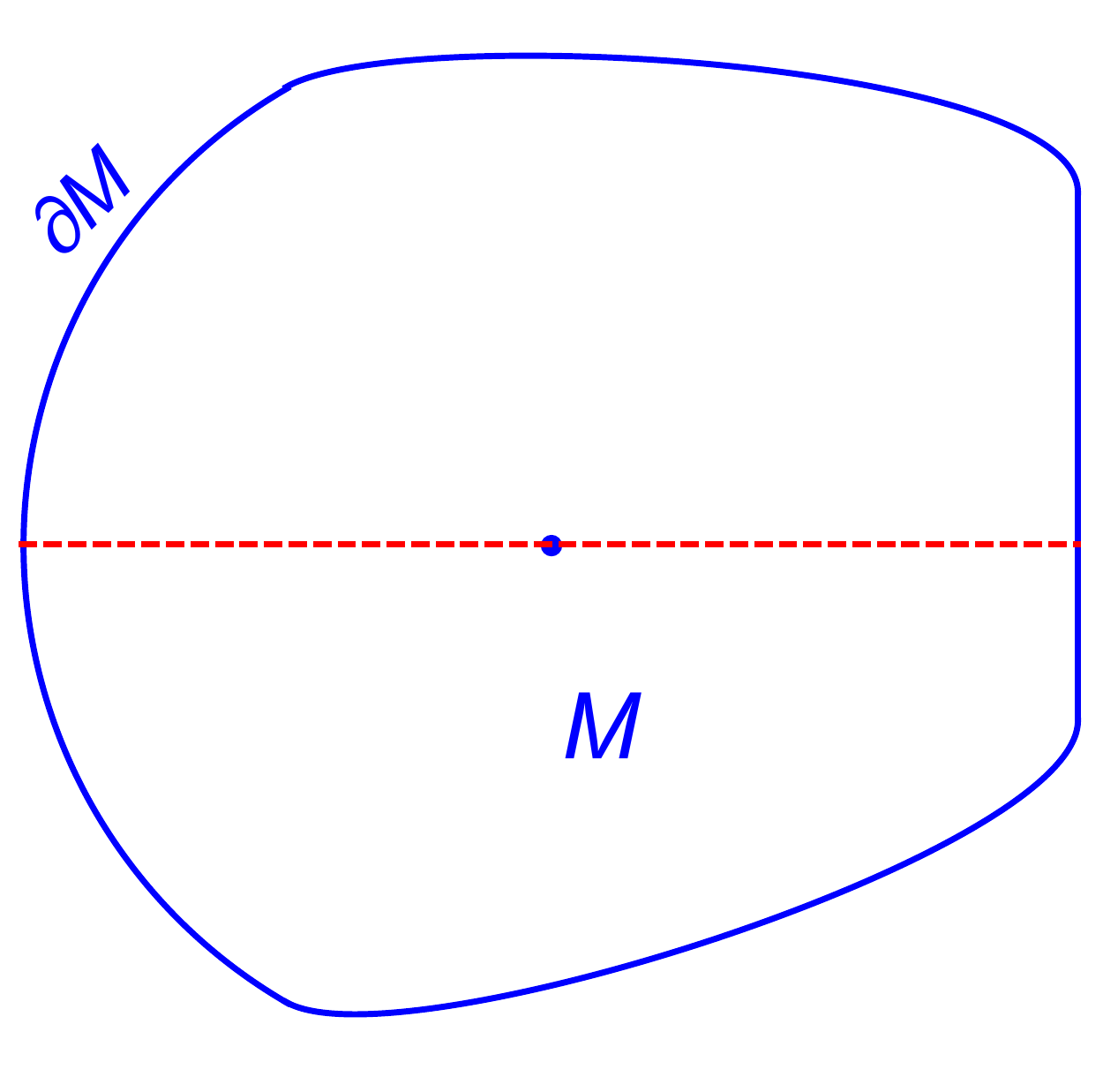, height=5cm} \caption{The picture illustrates an example of a non-degenerate OGC which has one $\partial M$-focal point, but its backward reparameterization has no $\partial M$-focal point. Here $M$ is a compact subset of a Euclidean space. The boundary of $M$ has a portion which is a spherical surface. An orthogonal geodesic chord departs from this surface, and it passes through the center of the sphere, which is a focal point. At the final endpoint, the geodesic meets orthogonally a portion of $\partial M$ which is planar (i.e., flat). Thus, there are no focal points along the backward reparameterization of the geodesic. Note that this phenomenon only occurs when the metric fails to be non-focal, so that the picture shows also that the non-focal property is \emph{not} generic. Namely,  any sufficiently small perturbation of the metric preserves the  existence of a non-degenerate OGC, which has different Morse indexes when run in the two directions, by the stability of the Morse index. Therefore, any perturbed metric fails to be non-focal.}
\label{fig:2}
\end{center}
\end{figure}
\begin{proof}
From Proposion~\ref{thm:secondorder},  part (a), it follows that if every OGC is nondegenerate, then also every critical point of $s_g$ is  nondegenerate, i.e., $s_g$ is a Morse function. Let us show that, given a critical point $p$ of $s_g$, then $q_p$ is a critical point with the same Morse index.
From Corollary~\ref{thm:MIT}, the difference between the Morse indices of $p$ and $q_p$ equals the difference of the number of $\partial M$-focal points along $\gamma_p$ and along $\gamma_{q_p}$ (which incidentally is the same geodesic\footnote{\label{foo:4}Observe that the fact that $\gamma_{q_p}$ is the same as $\gamma_p$ with the opposite orientation does not imply in principle that the number of $\partial M$-focal points along the two geodesics is the same, see Figure \ref{fig:2}.}
 as $\gamma_p$, with the opposite orientation, see Corollary~\ref{thm:pcritiffqcrit}). We claim that the number of $\partial M$-focal points along any geodesic $\gamma_r:[0,s_g(r)]\to M$ is constant for al $r\in\partial M$. This follows easily from the assumption that no $\gamma_r$ has any $\partial M$-focal point in $M$, and an argument of stability of the number of focal points in Riemannian geometry, see \cite{MerPicTau02} for details. Namely, as $r$ runs in $\partial M$, the number of focal points (counted with multiplicity) along $\gamma_r$ changes continuously, and by connectedness, if such number is not constant in $\partial M$, then there would be a $\partial M$-focal point in $\partial M$ along some $\gamma_r$. This concludes the proof.
\end{proof}
\begin{rem}  It is not hard to show that in all the results of this section, the assumptions cannot be relaxed.
As we observed in Remark~\ref{rem:counterex1}, Figure \ref{fig:01} shows an example of a metric which is not non-focal, and where the critical points of $s_g$ do not correspond to OGC's (cf.\ Proposition~\ref{thm:critpointssg} and Corollary~\ref{thm:pcritiffqcrit}). Figure \ref{fig:2} illustrates an example where the function $s_g$ is a Morse function, but not even-Morse (cf. Corollary~\ref{thm:Morse+equaindex}). Figure \ref{fig:3} gives an example of a 
regular and non-focal metric on a manifold of arbitrary dimension, whose boundary is not connected, and admitting only $2$ OGC's.
\end{rem}
\begin{figure}
\begin{center}
\psfull \epsfig{file=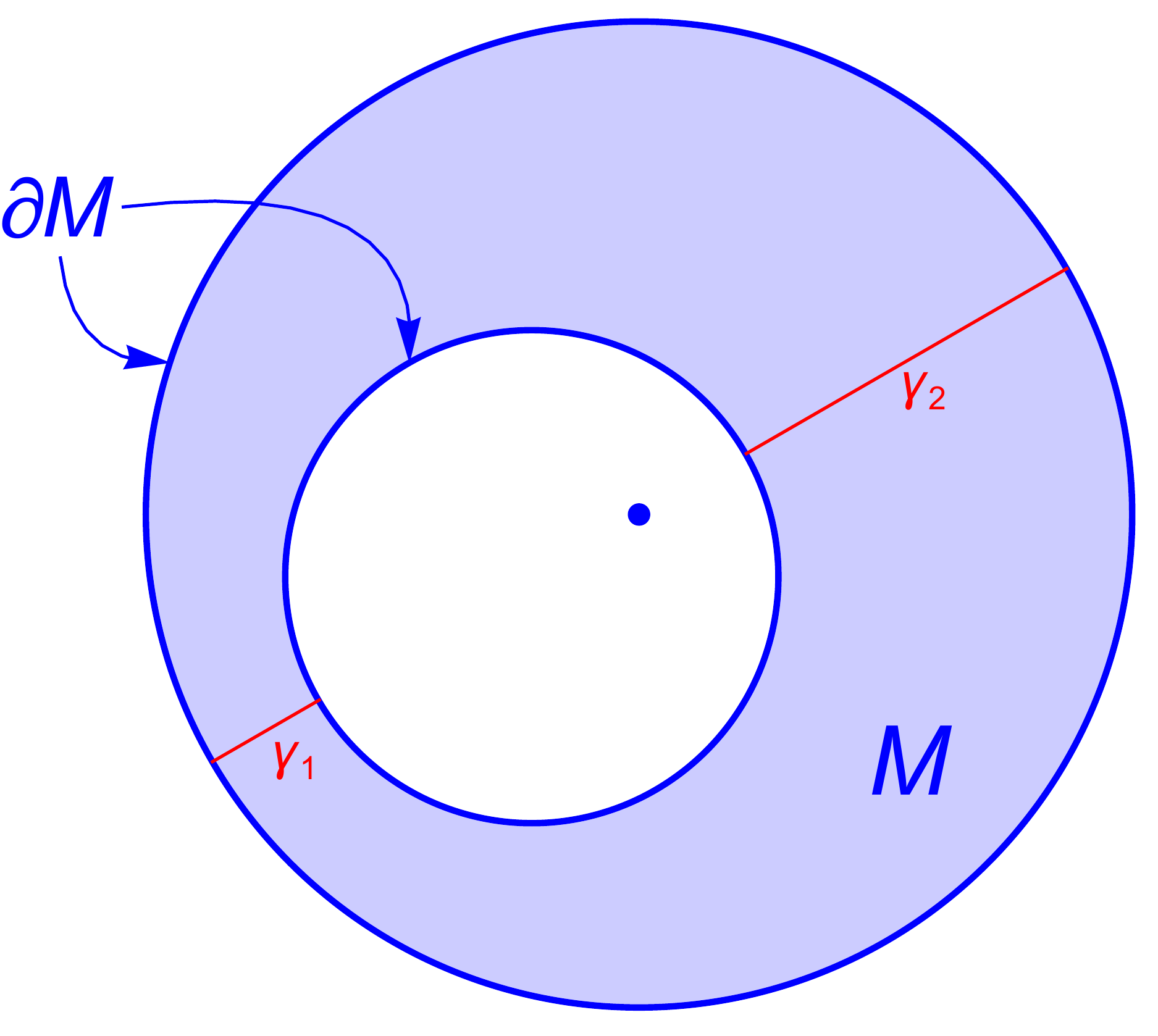, height=5cm} \caption{When the boundary of $M$ is not connected, one cannot expect the existence of more than $2$ orthogonal geodesic chords, regardless of the dimension of $M$. Consider for instance an annular region of $\mathds R^{m+1}$ bounded by two non-concentrical spheres. In this situation, the metric is regular and, if the inner sphere contains the center of the outer sphere in its interior, then the metric is also non-focal. Thus, $s_g$ is an 
even-Morse function on a non-connected manifold. However, there are only two orthogonal geodesic chords, and thus $s_g$ admits only two pairs of critical points with the same Morse index. Observe that in Remark~\ref{rem:intro} on page~\pageref{rem:intro} only connected manifolds were considered.}
\label{fig:3}
\end{center}
\end{figure}
\begin{prop}\label{thm:OGCsphere}
Under the assumptions of Corollary~\ref{thm:Morse+equaindex}, there are at least $(m+1)$ distinct orthogonal geodesic chords in $(B^{m+1},g)$.
\end{prop}
\begin{proof}
It follows immediately from Theorem~A and Corollary~\ref{thm:Morse+equaindex}.
\end{proof}
In next section we will discuss the assumptions of the results proved above, and we will prove the genericity of the nondegeneracy condition for OGC's.
\end{section}

\begin{section}{Proof of Theorem~B. A discussion of the assumptions.}
\subsection{On the genericity of metrics without degenerate OGC's}
Recall that a metric on a compact Riemannian manifold is called \emph{bumpy} if every one of its closed geodesic (possibly iterated)
is nondegenerate, i.e., it does not admit non-trivial periodic Jacobi fields. A classical result due to Anosov \cite{Ano82} asserts that bumpy metric are generic. Genericity of metrics all of whose geodesics satisfying more general boundary condition are nondegenerate is proved in \cite{BetGia2010}.

Let us introduce the following terminology. Given a compact manifold $M$ with boundary, a metric $g\in\mathrm{Reg}_*(M)$ will be called
\emph{$\partial$-bumpy} if every OGC in $M$ is nondegenerate.

The following is what we need for the conclusion of Theorem~B:
\begin{prop}\label{thm:maingenericity}
Let $M$ be a compact manifold with boundary. The set:
\begin{equation}\label{eq:Reg_generic}
\widetilde{\mathrm{Reg}}(M):=\Big\{g\in\mathrm{Reg}_*(M):\text{$g$ is $\partial$-bumpy}\Big\}
\end{equation}
is $C^k$-generic in $\mathrm{Reg}_*(M)$ for all $k=2,\ldots,+\infty$.
\end{prop}
\begin{proof}
It suffices to show that $\widetilde{\mathrm{Reg}}(M)$ is $C^k$-generic in the set of all metrics on $M$.
This is an application of the result in \cite{BetGia2010}. In order to obtain the desired conclusion, one must observe that the proof of \cite[Theorem~5.10]{BetGia2010} yields the genericity of metrics having only nondegenerate geodesics satisfying arbitrary boundary conditions, provided that one can exclude \textsl{a priori} the existence of \emph{strongly degenerate}\footnote{The notion of strong degeneracy for a periodic geodesic has been introduced in \cite{BetGia2010}, and it is stated in terms of the existence of a non-trivial periodic Jacobi field along an iterated periodic geodesic, satisfying some additional properties. A detailed discussion of this notion is irrelevant in the context of the present paper.} periodic geodesics in this class. This is always the case for orthogonal geodesic chords, because this class does not contain any periodic geodesic at all.
\end{proof}
\subsection{${\pmb\partial}{\pmb M}$-focal points and non-focal metrics}
\label{sub:3.2}
Unlike bumpy metrics, the set of non-focal Riemannian metrics on a compact manifold with boundary $\partial M$
is not generic in general.
\begin{example}
For instance, one can construct examples of hypersurfaces $\Sigma$ in $\mathds R^n$, diffeomorphic to an $(n-1)$-disk, whose set of focal points contains another hypersurface $\mathcal F$. Now, consider the closure $\overline\Omega$ of a bounded open subset whose smooth boundary $\partial\Omega$ contains $\Sigma$, and such that $\partial\Omega$ intercepts $\mathcal F$ transversally. In this situation, transversality implies that any $C^1$-perturbation of the flat metric on $\overline\Omega$ will produce $\partial\Omega$-focal points along $\partial\Omega$.
The situation is illustrated concretely in Figure \ref{fig:2}.
The picture provides an example of a metric which fails to be non-focal,
and the same holds also for small perturbations of the metric.
\end{example}
However, the absence of $\partial M$-focal points on $\partial M$ is still a quite general assumptions, that holds in a large variety of circumstances. Intuitively, concavity of the boundary (which is a property satisfied by Jacobi metric) and non-positive curvature have an effect of keeping focal points away from $\partial M$. More precisely, we can prove that when the sectional curvature of $g$ and the eigenvalues of the shape operator of $\partial M$ are not too large compared to the \emph{size} of $M$, then there are no $\partial M$-focal points in $M$.
In order to give basis to this assertion, let us introduce the following constants.

Let $g$ be a regular metric on $M$; the maximum of the function $s_g$ on $\partial M$, denoted by $L(M,g)$ will be called the \emph{maximal crossing time} of $(M,g)$. Clearly, $L(M,g)$ is somewhat related to, however in general larger than, the diameter of $(M,g)$. Moreover, denote by $\Lambda$ the maximum, for $p\in\partial M$, of the largest eigenvalue of $\mathcal S_p$, and by $K$ the maximum value of the sectional curvatures
of all $2$-planes tangent to $M$.
\begin{prop}\label{thm:nofocalpoints}
Let $M$ be a compact manifold with smooth boundary.
\begin{itemize}
\item[(a)] The set of non-focal metrics on $M$ is open with respect to the $C^2$-topology.
\item[(b)] A sufficient condition for a metric $g$ to be non-focal is that the following inequality holds:
\begin{equation}\label{eq:suffcondnonfocality}
K^+\cdot L(M,g)^2+\Lambda^+\cdot L(M,g)-\tfrac14<0,
\end{equation}
\end{itemize}
(where $a^+$ denotes the positive part of the real numner $a$.)
Condition \eqref{eq:suffcondnonfocality} implies indeed that there is no focal point to $\partial M$ in $M$.
\end{prop}
\begin{proof}
The set of focal points to a given compact hypersurface along a normal geodesic is stable by $C^2$-perturbations of all data. A formal proof of this fact can be found in reference~\cite{MerPicTau02}. From this stability, (a) follows easily.

For the proof of (b), let us show that, under the assumption \eqref{eq:suffcondnonfocality}, given any $p\in\partial M$, the index form $I_p$ is positive definite on the space on the vector space $\mathcal U_p$ of (piecewise smooth) vector fields $V$ along $\gamma_p$ satisfying $V(0)\in T_p(\partial M)$, $V\big(s_g(p)\big)=0$, and $g(V,\dot\gamma_p)\equiv0$.
Recall that $\gamma_p:\big[0,s_g(p)\big]\to M$ is a unit speed geodesic.

Let $V\in\mathcal U_p\setminus\{0\}$ be fixed; up to normalization, we can assume:
\begin{equation}\label{eq:normalization}
\int_0^{s_g(p)}g(V,V)\,\mathrm ds=1.
\end{equation}
Since $V\big(s_g(p)\big)=0$, for all $t\in\big[0,s_g(p)\big]$ we have:
\[-g\big(V(t),V(t)\big)=2\int_t^{s_g(p)}g(V',V)\,\mathrm ds,\]
which gives
\begin{equation}\label{eq:F1}
g\big(V(t),V(t)\big)\le2\int_0^{s_g(p)}g(V',V')^\frac12g(V,V)^\frac12\,\mathrm ds.
\end{equation}
Integrating \eqref{eq:F1} on $[0,s_g(p)]$ and keeping in mind \eqref{eq:normalization}, Schwarz's inequality gives:
\begin{equation}\label{eq:F2}
1=\int_0^{s_g(p)}g(V,V)\,\mathrm ds\le 4s_g(p)^2\int_0^{s_g(p)}g(V',V')\,\mathrm ds.
\end{equation}
Setting $t=0$ in \eqref{eq:F1} and using again Schwarz's inequality, we obtain:
\begin{equation}\label{eq:F3}
g\big(V(0),V(0)\big)\le4s_g(p)\int_0^{s_g(p)}g(V',V')\mathrm ds.
\end{equation}
Let us now estimate $I_p(V,V)$, using formula \eqref{eq:indexform}; note that $g\big(\dot\gamma_p,V)\dot\gamma_p,V\big)$
is equal to $-K(\dot\gamma_p,V)g(V,V)$, where $K(\dot\gamma_p,V)$ is the sectional curvature of the $2$-plane spanned by $\dot\gamma_p$ and $V$.
Then, using the very definition of the constants $\Lambda$ and $K$, we get:
\begin{multline}\label{eq:finalestimateIp}
I_p(V,V)\ge-\Lambda g\big(V(0),V(0)\big)+\int_0^{s_g(p)}\big[g(V',V')-Kg(V,V)\big]\,\mathrm ds\\
\stackrel{\text{by \eqref{eq:F3}}}\ge-4\Lambda^+ s_g(p)\int_0^{s_g(p)}g(V',V')\,\mathrm ds+\int_0^{s_g(p)}\big[g(V',V')-Kg(V,V)\big]\,\mathrm ds\\
\stackrel{\text{by \eqref{eq:F2}}}\ge\Big[-4\Lambda^+ s_g(p)+1-4K^+s_g(p)^2\Big]\cdot\int_0^{s_g(p)}g(V',V')\,\mathrm ds.
\end{multline}
Note that if $V\ne0$, then $\int_0^{s_g(p)}g(V',V')\,\mathrm ds>0$, because $V\big(s_g(p)\big)=0$, and therefore $V$ cannot be parallel along $\gamma$. Thus, from \eqref{eq:finalestimateIp} it follows that $I_p(V,V)>0$ if:
\begin{equation}\label{eq:F4}
1-4\Lambda^+ s_g(p)-4K^+s_g(p)^2>0.
\end{equation}
Note that \eqref{eq:suffcondnonfocality} coincides with \eqref{eq:F4} when $s_g(p)=L(M,g)$.
On the other hand, if \eqref{eq:suffcondnonfocality} is satisfied, it is easy to see that
\eqref{eq:F4} holds for every $p\in\partial M$.
\end{proof}

For instance, when $M$ is a ball, from part (a) of Proposition~\ref{thm:nofocalpoints} we deduce the following:
\begin{cor}\label{thm:perturbradsym}
If $M$ is diffeomorphic to a ball, then the set of regular and non-focal metrics on $M$ is a $C^2$-open set that contains all the rotationally symmetric metrics.
\end{cor}
\begin{proof}
A rotationally symmetric metric in the ball is regular, because the geodesics starting orthogonally to the boundary are radial, and therefore they arrive transversally (orthogonally) to the boundary at the other endpoint. Moreover, for all these metrics the unique focal point to the boundary is the center of symmetry, which is far from the boundary.
Hence, sufficiently small perturbations of rotationally symmetric metrics are non-focal.
\end{proof}

\end{section}


\begin{thebibliography}{20}
\bibitem{Ano82} \textsc{D. V. Anosov}, \emph{Generic properties of closed geodesics}, Izv.\ Akad.\ Nauk SSSR Ser.\ Mat.\ \textbf{46} (1982), no.\ 4, 675--709.

\bibitem{BetGia2010} \textsc{R. G. Bettiol, R. Giamb\`o}, \emph{Genericity of nondegenerate geodesics with general boundary conditions}, Topol.\ Methods Nonlinear Anal.\ \textbf{35} (2010), no.\ 2, 339--365.

\bibitem{bos} \textsc{W.\ Bos}, \emph{Kritische Sehenen auf Riemannischen
Elementarraumst\"ucken}, Math.\ Ann.\ {\bf 151} (1963), 431--451.


\bibitem{arma} \textsc{R. Giamb\`o, F. Giannoni, P. Piccione}, \emph{Multiple Brake Orbits and Homoclinics in Riemannian Manifolds}, Arch. Rational Mechanics and Analysis  \textbf{200}, Issue 2 (2011), 691--724.


\bibitem{cag} \textsc{R. Giamb\`o, F. Giannoni, P. Piccione}, \emph{Morse Theory for geodesics in singular conformal metrics}
Comm.\ Anal.\ Geom.\ \textbf{22} (2014), 779--809.

\bibitem{JDE} \textsc{R. Giamb\`o, F. Giannoni, P. Piccione}, \emph{Examples with minimal number of brake orbits and homoclinics in annular potential regions}, J.\ Diff.\ Equations \textbf{256} (2014),  2677--2690.

\bibitem{pre} \textsc{R. Giamb\`o, F. Giannoni, P. Piccione}, \emph{Multiple brake orbits in $m$-dimensional disks}, preprint (2015) arXiv:1503.05805

\bibitem{LiuLong} \textsc{H.\ Liu, Y.\ Long}, {\em Resonance identity for symmetric closed characteristics on symmetric convex Hamiltonian energy hypersurfaces and its applications}, J. Differential Equations \textbf{255} (2013) 2952--2980.

\bibitem{LZ} \textsc{C.\ Liu, D.\ Zhang}, \emph{Seifert conjecture in the even convex case}, Comm.\ Pure and Applied Math.\ \textbf{67} (2014) 1563--1604.

\bibitem{LZZ} \textsc{Y. Long,  D. Zhang, C. Zhu}, \emph{Multiple brake orbits in bounded convex symmetric domains},
Adv.\ Math. \textbf{203} (2006), no.\ 2, 568--635.

\bibitem{Long} \textsc{Y. Long,  C. Zhu}, \emph{Closed characteristics on
compact convex hypersurfaces in $\mathds{R}\sp{2n}$}
Ann.\ of Math.\ (2) {\bf155}  (2002),  no.\ 2, 317--368.

\bibitem{LustSchn} \textsc{L.\ Lusternik, L.\ Schnirelman}, \emph{Methodes
Topologiques dans les Problemes Variationelles}, Hermann, 1934.

\bibitem{MerPicTau02} \textsc{F. Mercuri, P. Piccione, D. V. Tausk}, \emph{Stability of the conjugate index, degenerate conjugate points and the Maslov index in semi-Riemannian geometry},
Pacific J. Math.\ \textbf{206} (2002), no.\ 2, 375--400.

\bibitem{Milnor_Morse} \textsc{J. Milnor}, \emph{Morse theory}. Annals of Mathematics Studies, No.\ 51 Princeton University Press, Princeton, N.J. 1963.


\bibitem{PicTau99} \textsc{P. Piccione, D. V. Tausk}, \emph{A note on the Morse index theorem for geodesics between submanifolds in semi-Riemannian geometry}, J.\ Math.\ Phys.\ \textbf{40} (1999), no.\ 12, 6682--6688.

\bibitem{rab}
\textsc{P.\ H.\ Rabinowitz}, \emph{Critical point theory and applications to differential equations: a survey}. Topological nonlinear analysis, 464--513, Progr. Nonlinear Differential Equations Appl., 15, Birkhauser Boston, Boston, MA, 1995.

\bibitem{seifert} \textsc{H.\ Seifert}, \emph{Periodische Bewegungen Machanischer
Systeme}, Math.\ Z.\ {\bf51} (1948), 197--216.

\bibitem{Z1} \textsc{D. Zhang}, \emph{Brake type closed characteristics on reversible compact convex
hypersurfaces in $\mathds{R}^{2n}$}, Nonlinear Analysis 74 (2011) 3149--3158.

\bibitem{Z2} \textsc{D. Zhang, C.Liu}, \emph{Multiplicity of brake orbits on compact convex symmetric
reversible hypersurfaces in $\mathds{R}^{2n}$ for $n \geq 4$},
Proc. London Math. Soc. (3) 107 (2013) 1--38

\bibitem{Z3} \textsc{D.Zhang, C.Liu}, \emph{Multiple brake orbits on compact convex symmetric reversible hypersurfaces in $\mathds{R}^{2n}$}, Ann.\ Inst.\ H. Poincar\'e Anal. Non Lin\'eaire \textbf{31} (2014), no.\ 3, 53--554.


\end{thebibliography}
\end{document}